\documentclass[11pt, oneside]{amsart}

\usepackage{etex}
\usepackage{graphicx}
\usepackage{amsmath,amsthm}
\usepackage{amsfonts}
\usepackage{amssymb}
\usepackage{hyperref}
\usepackage{array}
\usepackage{multirow}
\usepackage{longtable}
\usepackage{setspace}
\usepackage{pgf,tikz}
\usepackage{enumerate}
\usepackage{multicol}
\usepackage[letterpaper, total={6in, 8in}]{geometry}

\usetikzlibrary{shapes.geometric}

\usepackage{mathrsfs}
\usetikzlibrary{arrows}

\usetikzlibrary{decorations.markings}
\tikzset{->-/.style={decoration={
  markings,
  mark=at position #1 with {\arrow{>}}},postaction={decorate}}}

\usetikzlibrary{calc}
\usetikzlibrary{graphs}
\usetikzlibrary{positioning}

\newtheorem{theorem}{Theorem}
\newtheorem{proposition}[theorem]{Proposition}
\newtheorem{corollary}[theorem]{Corollary}

\theoremstyle{definition}

\newcommand{\blackvertex}[1]{\fill (#1) circle (2pt);}
\newcommand{\whitevertex}[1]{\fill (#1) circle (2pt);\fill[white] (#1) circle (1.5pt);}

\usepackage{color}
\definecolor{blue}{rgb}{0,0,1}

\newcolumntype{I}{>{\centering\arraybackslash} m{.15\linewidth}} 
\newcolumntype{A}{>{\centering\arraybackslash} m{.18\linewidth}} 
\newcolumntype{V}{>{\centering\arraybackslash} m{.22\linewidth}} 
\newcolumntype{W}{>{\centering\arraybackslash} m{.28\linewidth}}

\newcounter{confignum}	
\newcounter{casenum}	
\newcounter{subcasenum}	
\numberwithin{subcasenum}{casenum}
\newcounter{subsubcasenum}	
\numberwithin{subsubcasenum}{subcasenum}

\newcounter{stagenum}

\newcommand{\chunion}[2]{(#1,#2)}

\def\mad{\operatorname{Mad}}

\begin{document}

\title{Choosability with union separation}
\author{
Mohit Kumbhat$^1$
\and
Kevin Moss$^1$
\and
Derrick Stolee$^{1,2}$
}

\begin{abstract}
List coloring generalizes graph coloring by requiring the color of a vertex to be selected from a list of colors specific to that vertex.
One refinement of list coloring, called \emph{choosability with separation}, requires that the intersection of adjacent lists is sufficiently small.
We introduce a new refinement, called \emph{choosability with union separation}, where we require that the union of adjacent lists is sufficiently large.
For $t \geq k$, a $(k,t)$-list assignment is a list assignment $L$ where $|L(v)| \geq k$ for all vertices $v$ and $|L(u)\cup L(v)| \geq t$ for all edges $uv$.
A graph is \emph{$(k,t)$-choosable} if there is a proper coloring for every $(k,t)$-list assignment.
We explore this concept through examples of graphs that are not $(k,t)$-choosable, demonstrating sparsity conditions that imply a graph is $(k,t)$-choosable, and proving that all planar graphs are $(3,11)$-choosable and $(4,9)$-choosable.
\end{abstract}

\maketitle

\footnotetext[1]{Department of Mathematics, Iowa State University, Ames, IA, U.S.A. \texttt{$\{$mkumbhat,kmoss,dstolee$\}$@iastate.edu}}
\footnotetext[2]{Department of Computer Science, Iowa State University, Ames, IA, U.S.A.}


\section{Introduction}

For a graph $G$ and a positive integer $k$, a \emph{$k$-list assignment} of $G$ is a function $L$ on the vertices of $G$ such that $L(v)$ is a set of size at least $k$.
An \emph{$L$-coloring} is an assignment $c$ on the vertices of $G$ such that $c(v) \in L(v)$ for all vertices $v$ and $c(u) \neq c(v)$ for all adjacent pairs $uv$.
A graph is \emph{$k$-choosable} if there exists an $L$-coloring for every $k$-list assignment $L$ of $G$, and $G$ is \emph{$k$-colorable} if there exists an $L$-coloring for the $k$-list assignment $L(v) = \{1,\dots,k\}$. The minimum $k$ for which $G$ is $k$-choosable is called the \emph{choosability} or the \emph{list-chromatic number} of $G$ and is denoted by $\chi_\ell(G)$.
Erd\H{o}s, Rubin, and Taylor~\cite{erdos1979choosability} and independently Vizing~\cite{Vizing} introduced the concept of list coloring and demonstrated that there exist graphs that are $k$-colorable but not $k'$-choosable for all $k' \geq k \geq 2$.
Since its introduction, choosability has received significant attention and has been refined in many different ways.



One refinement of choosability is called \emph{choosability with separation} and has received recent attention~\cite{DMWS14,choi,furedi,kl,skrekovski} since it was defined by Kratochv\'il, Tuza, and Voigt~\cite{krat}.
Let $G$ be a graph and let $s$ be a nonnegative integer called the \emph{separation} parameter.
A \emph{$(k,k-s)$-list assignment} is a $k$-list assignment $L$ such that $|L(u)\cap L(v)| \leq k-s$ for all adjacent pairs $uv$.
We say a graph $G$ is \emph{$(k,t)$-choosable} if, for any $(k,t)$-list assignment $L$, there exists an $L$-coloring of $G$.
As the separation parameter $s$ increases, the restriction on the intersection-size of adjacent lists becomes more strict.

We introduce a complementary refinement of choosability called \emph{choosability with union separation}.
A \emph{$(k,k+s)$-list assignment} is a $k$-list assignment $L$ such that $|L(u)\cup L(v)| \geq k+s$ for all adjacent pairs $uv$.
We similarly say $G$ is $(k,t)$-choosable to imply choosability with either kind of separation, depending on $t \leq k$ or $k < t$.
Observe that if $G$ is $(k,k+s)$-choosable, then $G$ is both $(k,k-r)$-choosable and $(k,k+r)$-choosable for all $r \geq s$.
Note that if $L$ is a $(k,k-s)$-list assignment, we may assume that $|L(v)| = k$ as removing colors from lists does not violate the intersection-size requirement for adjacent vertices.
However, when considering a $(k,k+s)$-list assignment, we may not remove colors from lists as that may violate the union-size requirement for adjacent vertices.
Due to this asymmetry, we do not know if there is a function $f(k,s)$ such that every $(k,k-s)$-choosable graph is also $(k,k+f(s))$-choosable.

Thomassen~\cite{thomassen} proved that all planar graphs are $5$-choosable.
The main question we consider regarding planar graphs and choosability with union separation is identifying minimum integers $t_3$ and $t_4$ such that all planar graphs are $(3,t_3)$-choosable and $(4,t_4)$-choosable.
We demonstrate that $6 \leq t_3 \leq 11$ and $6 \leq t_4 \leq 9$.

Kratochv\'il, Tuza, and Voigt~\cite{krat2} proved that all planar graphs are $(4,1)$-choosable and conjecture that all planar graphs are $(4,2)$-choosable.
Voigt~\cite{voigt1993list} constructed a planar graph that is not $(4,3)$-choosable and hence is not $(4,5)$-choosable.
We show that $t_4 \leq 9$.

\begin{theorem}\label{thm:49choosable}
All planar graphs are $\chunion{4}{9}$-choosable.
\end{theorem}

A \emph{chorded $\ell$-cycle} is a cycle of length $\ell$ with one additional edge.
For each $\ell \in \{5,6,7\}$, Berikkyzy \emph{et al.}~\cite{DMWS14} demonstrated that if $G$ is a planar graph that does not contain a chorded $\ell$-cycle, then $G$ is $(4,2)$-choosable.
The case $\ell = 4$ is notably missing from their results, especially since Borodin and Ivanova~\cite{borodin} proved that if $G$ is a planar graph that does not contain a chorded 4-cycle or a chorded 5-cycle, then $G$ is 4-choosable.
We prove that if $G$ is a planar graph containing no chorded 4-cycle, then $G$ is $(4,7)$-choosable (see Theorem~\ref{thm:47choosable}).

Kratochv\'il, Tuza, and Voigt~\cite{krat2} conjecture that all planar graphs are $(3,1)$-choosable.
Voigt~\cite{voigt} constructed a planar graph that is not $(3,2)$-choosable and hence is not $(3,4)$-choosable.
In Section~\ref{sec:counterexamples} we construct graphs that are not $(k,t)$-choosable, including a planar graph that is not $(3,5)$-choosable.
This hints towards a strong difference between intersection separation and union separation.
We show that $t_3 \leq 11$.

\begin{theorem}\label{thm:311choosable}
All planar graphs are $\chunion{3}{11}$-choosable.
\end{theorem}

We also consider sparsity conditions that imply $(k,t)$-choosability.
For a graph $G$, the \emph{maximum average degree} of $G$, denoted $\mad(G)$, is the maximum fraction $\frac{2|E(H)|}{|V(H)|}$ among subgraphs $H \subseteq G$.
If $\mad(G) < k$, then $G$ is $(k-1)$-degenerate and hence is $k$-choosable. 
Since $\mad(K_{k+1}) = k$ and $\chi_\ell(K_{k+1})>k$, this bound on $\mad(G)$ cannot be relaxed.
In Section~\ref{sec:sparse}, we prove that  $G$ is $(k,t)$-choosable when $\mad(G) < 2k - o(1)$ where $o(1)$ tends to zero as $t$ tends to infinity.
This is asymptotically sharp as we construct graphs that are not $(k,t)$-choosable with $\mad(G) = 2k-o(1)$.

Many of our proofs use the discharging method.
For an overview of this method, see the surveys of Borodin~\cite{BorodinSurvey}, Cranston and West~\cite{CW}, or the overview in Berikkyzy \emph{et al.}~\cite{DMWS14}.
We use a very simple reducible configuration that is described by Proposition~\ref{reducible-1} in Section~\ref{sec:reducible}.

\subsection{Notation}

A (simple) graph $G$ has vertex set $V(G)$ and edge set $E(G)$. Additionally, if $G$ is a plane graph, then $G$ has a face set $F(G)$.
Let $n(G) = |V(G)|$ and $e(G) = |E(G)|$.
For a vertex $v \in V(G)$, the set of vertices adjacent to $v$ is  the \emph{neighborhood} of $v$, denoted $N(v)$.
The \emph{degree} of $v$, denoted $d(v)$, is the number of vertices adjacent to $v$.
We say $v$ is a $k$-vertex if $d(v) = k$, a $k^-$-vertex if $d(v) \leq k$ and a $k^+$-vertex if $d(v) \geq k$.
Let $G - v$ denote the graph given by deleting the vertex $v$ from $G$.
For an edge $uv \in E(G)$, let $G -uv$ denote the graph given by deleting the edge $uv$ from $G$.
For a plane graph $G$ and a face $f$, let $\ell(f)$ denote the length of the face boundary walk; say $f$ is a $k$-face if $\ell(f) = k$ and a $k^+$-face if $\ell(f) \geq k$.


\section{Non-$(k,t)$-Choosable Graphs}\label{sec:counterexamples}

\begin{proposition}\label{prop:counterexample}
For all $t \geq k \geq 2$, there exists a bipartite graph that is not $(k,t)$-choosable.
\end{proposition}

\begin{proof}
Let $u_1,\dots,u_k$ be nonadjacent vertices and let $L(u_1),\dots,L(u_k)$ be disjoint sets of size $t-k+1$.
For every element $(a_1,\dots,a_k) \in \prod_{i=1}^k L(u_i)$, let $A = \{a_1,\dots,a_k\}$, create a vertex $x_A$ adjacent to $u_i$ for all $i \in [k]$, and let $L(x_A) = A$ (see Figure~\ref{fig:book}).
Notice that  $|L(u_i)\cup L(x_A)| = t$ for all $i \in [k]$ and all vertices $x_A$, so $L$ is a $(k,t)$-list assignment.
If there is a proper $L$-coloring $c$ of this graph, then let $A = \{ c(u_i) : i \in [k]\}$; the color $c(x_A)$ is in $A$ and hence the coloring is not proper.
\end{proof}

\begin{figure}[htp]
\centering
\begin{tikzpicture}[vtx/.style={shape=coordinate}]
	\node[vtx] (v0) at (-2.0000,1.0000) {};
	\node[vtx] (v1) at (-1.0000,1.0000) {};
	\node[vtx] (v2) at (0.0000,1.0000) {};
	\node[vtx] (v3) at (1.0000,1.0000) {};
	\node[vtx] (v4) at (2.0000,1.0000) {};
	\node[vtx] (v5) at (-3.5000,0.0000) {};
	\node[vtx] (v6) at (-2.5000,0.0000) {};
	\node[vtx] (v7) at (-1.5000,0.0000) {};
	\node[vtx] (v8) at (-0.5000,0.0000) {};
	\node[vtx] (v9) at (0.5000,0.0000) {};
	\node[vtx] (v10) at (1.5000,0.0000) {};
	\node[vtx] (v11) at (2.5000,0.0000) {};
	\node[vtx] (v12) at (3.5000,0.0000) {};
	\draw (v0) -- (v5);
	\draw (v0) -- (v6);
	\draw (v0) -- (v7);
	\draw (v0) -- (v10);
	\draw (v0) -- (v11);
	\draw (v0) -- (v12);
	\draw (v1)  (v5);
	\draw (v1) -- (v6);
	\draw (v1) -- (v7);
	\draw (v1) -- (v10);
	\draw (v1) -- (v11);
	\draw (v1) -- (v12);
	\draw (v2) -- (v5);
	\draw (v2) -- (v6);
	\draw (v2) -- (v7);
	\draw (v2) -- (v10);
	\draw (v2) -- (v11);
	\draw (v2) -- (v12);
	\draw (v3) -- (v5);
	\draw (v3) -- (v6);
	\draw (v3) -- (v7);
	\draw (v3) -- (v10);
	\draw (v3) -- (v11);
	\draw (v3) -- (v12);
	\draw (v4) -- (v5);
	\draw (v4) -- (v6);
	\draw (v4) -- (v7);
	\draw (v4) -- (v10);
	\draw (v4) -- (v11);
	\draw (v4) -- (v12);
	\foreach \p in {v0,v1,v2,v3,v4,v5,v6,v7,v10,v11,v12}
	{
		\blackvertex{\p}
	}
	\foreach \p in {-0.5,0,0.5}
	{
		\fill (\p,0) circle (1.5pt);
	}
	
	\draw[<->] (-3.5,-0.5)--(3.5,-0.5);
	\node[vtx,label=below:{$(t-k+1)^k$ vertices}] (label) at (0,-0.5) {};
	\draw[<->] (-2,1.5)--(2,1.5);
	\node[vtx,label=above:{$k$ vertices}] (label) at (0,1.5) {};
\end{tikzpicture}
\caption{\label{fig:book}A graph that is not $(k,t)$-choosable.}
\end{figure}
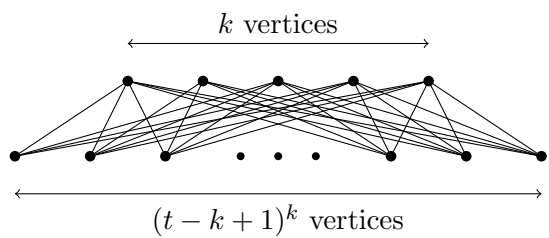

Observe that the graph constructed in Proposition~\ref{prop:counterexample} has average degree $\frac{2k(t-k+1)^k}{k+(t-k+1)^k}$; as $t$ increases, this fraction approaches $2k$ from below.
Observe that when $k=2$ the graph built in Proposition~\ref{prop:counterexample} is planar, giving us the following corollary.

\begin{corollary}
For all $t \geq 2$, there exists a bipartite planar graph that is not $(2,t)$-choosable.
\end{corollary}

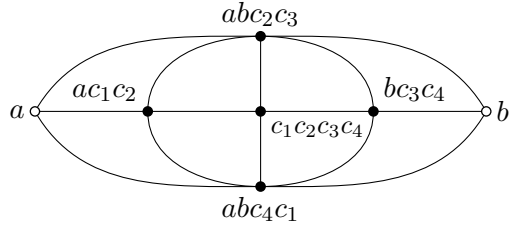
\begin{figure}[htp]
\centering
\begin{tikzpicture}[vtx/.style={shape=coordinate}]
	\node[vtx,label=left:{$a$}] (v0) at (-6.0000,0.0000) {};
	\node[vtx,label=right:{$b$}] (v1) at (0.0000,0.0000) {};
	\node[vtx,label=below:{$abc_4c_1$}] (v2) at (-3.0000,-1.0000) {};
	\node[vtx,label=above left:{$ac_1c_2$}] (v3) at (-4.5000,0.0000) {};
	\node[vtx,label=above:{$abc_2c_3$}] (v4) at (-3.0000,1.0000) {};
	\node[vtx,label=above right:{$bc_3c_4$}] (v5) at (-1.5000,0.0000) {};
	\node[vtx,label=below right:{\small $c_1c_2c_3c_4$}] (v6) at (-3.0000,0.0000) {};
	\draw (v0) to[out=-60,in=180] (v2);
	\draw (v0) -- (v3);
	\draw (v0) to[out=60,in=180] (v4);
	\draw (v1) to[out=-120,in=0] (v2);
	\draw (v1)  to[out=120,in=0] (v4);
	\draw (v1) -- (v5);
	\draw (v2) to[out=180,in=270] (v3);
	\draw (v2) -- (v6);
	\draw (v3) -- (v6);
	\draw (v4) to[out=0,in=90] (v5);
	\draw (v4) -- (v6);
	\draw (v5) -- (v6);
	\draw (v2) to[out=0,in=270] (v5);
	\draw (v3) to[out=90,in=180] (v4);
	\foreach \p in {v2,v3,v4,v5,v6}
	{
		\blackvertex{\p}
	}
	\foreach \p in {v0,v1}
	{
		\whitevertex{\p}
	}
\end{tikzpicture}
\caption{\label{fig:gadget35}A planar gadget with a $(3,5)$-list assignment.}
\end{figure}

We now construct a specific planar graph that is not $(3,5)$-choosable.

\begin{proposition}\label{prop:35counterexample}
There exists a planar graph that is not $(3,5)$-choosable.
\end{proposition}

\begin{proof}
Let $A$ and $B$ be disjoint sets of size three, and let $c_1,\dots,c_4$ be distinct colors not in $A \cup B$.
Let $v_A$ and $v_B$ be two vertices and let $L(v_A) = A$ and $L(v_B) = B$.
For each $a \in A$ and $b \in B$, consider the graph displayed in Figure~\ref{fig:gadget35}; create a copy of this graph where the left vertex is $v_A$ and the right vertex is $v_B$.
Assign lists to the interior vertices of this graph using the colors $\{a,b,c_1,\dots,c_4\}$ as shown in the figure.
Observe that $L$ is a $(3,5)$-list assignment.
If there exists a proper $L$-coloring, then let $a \in A$ be the color on $v_A$ and $b \in B$ be the color on $v_B$ and consider the copy of this gadget using these colors.
Observe that in the 4-cycle induced by the neighbors of the center vertex, all four colors $c_1,\dots,c_4$ must be present.
Then the coloring is not proper as the center vertex is assigned one of these colors.
\end{proof}



\section{Reducible Configurations}\label{sec:reducible}

To prove all of our main results, we consider a minimum counterexample and arrive at a contradiction through discharging. In this section, we describe the structures that cannot appear in a minimum counterexample.



\begin{proposition}\label{reducible-1}
Let $G$ be a graph, $uv$ an edge in $G$, $t \geq k \geq 3$, and $a = |N(u) \cap N(v)|$ with $a \in \{0,1,2\}$.
Let  $L$ be a $(k,t)$-list assignment and suppose that there exist $L$-colorings of $G - u$, $G - v$, and $G - uv$.
If $d(u) + d(v) \leq t + a$, then there exists an $L$-coloring of $G$.
\end{proposition}

\begin{proof}
If $|L(u)| > d(u)$, then the $L$-coloring of $G-u$ extends to an $L$-coloring of $G$ as there is a color in $L(u)$ that does not appear among the neighbors of $u$; thus we assume $|L(u)| \leq d(u)$.
By a symmetric argument we may assume $|L(v)| \leq d(v)$.
If $L(u) \cap L(v) = \varnothing$, then the $L$-coloring of $G-uv$ is an $L$-coloring of $G$; thus we assume $|L(u)\cap L(v)| \geq 1$ and $|L(u)\cup L(v)| \leq |L(u)| + |L(v)| - 1 \leq d(u) + d(v) - 1$.


Note that $t+2 \geq d(u)+d(v) \geq |L(u)|+|L(v)|\geq t+1$.
Thus $a \geq 1$ and either $|L(u)| = d(u)$ or $|L(v)| = d(v)$; assume by symmetry that $|L(u)|=d(u)$.
Let $c$ be an $L$-coloring of $G - u$.
For $x \in \{u,v\}$, let $L'(x)$ be the colors in $L(x)$ that do not appear among the neighbors $y \in N(x) \setminus \{u,v\}$.
Since $c(v) \in L'(v)$, we have $L'(v) \neq \varnothing$.
Since $|L(u)| = d(u)$, we have $L'(u) \neq \varnothing$.
Observe that $|L'(u) \cup L'(v)| \geq |L(u)\cup L(v)| - |N(u)\cup N(v)| + 2 \geq t - (d(u) + d(v) - a) + 2 \geq 2$.
Thus either $|L'(u)|=|L'(v)|=1$ or $|L'(x)| \geq 2$ for some $x \in \{u,v\}$ and therefore there are choices for $c'(u) \in L'(u)$ and $c'(v) \in L'(v)$ such that $c'(u) \neq c'(v)$.
If $c'(y) = c(y)$ for all $y \in V(G) \setminus \{u,v\}$, then $c'$ is an $L$-coloring of $G$.
\end{proof}

\section{Sparse Graphs}\label{sec:sparse}

In this section, we determine a relationship between sparsity and choosability with union separation.

\begin{theorem}\label{thm:sparse}
Let $k \geq 2$ and $t \geq 2k-1$.
If $G$ is a graph with $\mad(G) < 2k\left(1-\frac{k}{t+1}\right)$, then $G$ is $(k,t)$-choosable.
\end{theorem}

\def\finalcharge{\mathfrak{c}}
\begin{proof}
Let $\finalcharge = 2k-\frac{2k^2}{t+1}$.
Observe that since $t\geq 2k-1$ that $\finalcharge \geq k$.
For the sake of contradiction, suppose there exists a graph $G$ with $\mad(G) < \finalcharge$ and a $(k,t)$-list assignment $L$ such that $G$ is not $L$-choosable.
Select $(G,L)$ among such pairs to minimize $n(G) + e(G)$.
Observe that $k\leq |L(v)| \leq d(v)$ for every vertex.

We use discharging to demonstrate $\mad(G) \geq \finalcharge$, a contradiction.
Assign charge $d(v)$ to every vertex $v$, so the total charge sum is equal to $2e(G)$.
We discharge using the following rule:

\begin{quote}
(R) If $u$ is a vertex with $d(u) < \finalcharge$, then $u$ pulls charge $\frac{\finalcharge-d(u)}{d(u)}$ from each neighbor of $u$.
\end{quote}

Suppose that $v$ is a vertex that loses charge by (R).
Then there exists an edge $uv \in E(G)$ where $d(u) < \finalcharge$.
Note that since $G-uv$ is $L$-choosable, $|L(u)\cap L(v)| \geq 1$ and $d(u)+d(v) \geq t+1$. It follows that
\[
d(v) \geq t+1-d(u) > t+1-\finalcharge \geq \finalcharge.
\]
Therefore, a vertex either loses charge by (R) or gains charge by (R), not both.

Observe that if $d(u) < \finalcharge$, then $u$ pulls enough charge by (R) to end with charge at least $\finalcharge$.

Finally, suppose $v$ is a vertex with $d(v) = d \geq \finalcharge$.

If $d \geq t+1-k$, then neighbors of $v$ pull charge at most $\frac{\finalcharge-k}{k}$ from $v$.
The final charge on $v$ is given by
\[
	d - d\left(\frac{\finalcharge -k}{k}\right)
	= d\left(\frac{2k-\finalcharge}{k}\right)
	= d\left(\frac{2k-(2k-\frac{2k^2}{t+1})}{k}\right)
	= d\left(\frac{2k}{t+1}\right)
	\geq (t+1-k)\frac{2k}{t+1} = \finalcharge.
\]

Now suppose that $d < t+1-k$.
If a vertex $u$ pulls charge from $v$ by (R), then $d(u) \geq d' = t+1-d$.
Thus, $v$ loses charge at most $\frac{\finalcharge-d'}{d'}$ to each neighbor.
The final charge on $v$ is given by
\begin{align*}
	d - d\left(\frac{\finalcharge - d'}{d'}\right)
	&= d \left(1-\frac{\finalcharge-d'}{d'}\right)\\
	&= d \left(\frac{2d'-\finalcharge}{d'}\right)\\
	&= d \left(\frac{2(t+1-d)-\finalcharge}{t+1-d}\right)
\end{align*}
Observe that $d \left(\frac{2(t+1-d)-\finalcharge}{t+1-d}\right) \geq \finalcharge$ if and only if $2(t+1-d)d - (t+1)\finalcharge \geq 0$.
By the quadratic formula, this polynomial (in $d$) has roots at $d \in \left\{ \frac{1}{2}\left(t+1 \pm \sqrt{(t+1)(t+1-2\finalcharge)}\right)\right\}$; the discriminant is nonnegative since $(t+1)((t+1)-2\finalcharge) =  (t+1-2k)^2$.
Thus, the final charge on $v$ is below $\finalcharge$ if and only if $d < k$ or $d > t+1-k$, but we are considering $d$ where $k \leq \finalcharge \leq d < t+1-k$.
\end{proof}

Note that Theorem~\ref{thm:sparse} implies that a graph $G$ is $(4,15)$-choosable when $\mad(G) < 8\left(1-\frac{4}{16}\right) = 6$.
If $G$ is planar, then $\mad(G) < 6$ and hence is $(4,15)$-choosable.
There is no $t$ such that Theorem~\ref{thm:sparse} implies all planar graphs are $(3,t)$-choosable.
We now directly consider planar graphs and find smaller separations suffice.


\section{$\chunion{4}{t}$-choosability}
\label{sec:4choosability}

\begin{proof}[Proof of Theorem~\ref{thm:49choosable}.]
Suppose $G$ is a plane graph minimizing $n(G)+e(G)$ such that $G$ is not $L$-colorable for some $(4,9)$-list assignment $L$.
By minimality of $G$, we can assume that $d(v) \geq |L(v)| \geq 4$ for all vertices $v$ and $|L(u) \cap L(v)| \geq 1$ for all adjacent pairs $uv$.
By Proposition~\ref{reducible-1},  if $uv$ is an edge in $G$, then $d(u) + d(v) > 9 + \min(|N(u)\cap N(v)|,2)$.
Observe that $\min(|N(u)\cap N(v)|,2)$ is at least the number of 3-faces incident to the edge $uv$.

For each $v \in V(G)$ and $f \in F(G)$ define $\mu(v) = d(v) - 4$ and $\nu(f) = \ell(f) - 4$. Note that the total initial charge of $G$ is $-8$.
For a vertex $v$, let $t_3(v)$ be the number of 3-faces incident to $v$.
Apply the following discharging rule.

\begin{enumerate}[({R}1)] \itemsep0pc
\item If $v$ is a $5^+$-vertex and $f$ is an incident 3-face, then $v$ sends charge $\frac{\mu(v)}{t_3(v)}$ to $f$.
\end{enumerate}

All vertices and $4^+$-faces have nonnegative charge after applying (R1).

Let $f$ be a 3-face with  incident vertices $u,v,w$ where $d(u) \leq d(v) \leq d(w)$.
Since $\nu(f) = -1$, if suffices to show that $f$ receives charge at least 1 in total from $u$, $v$, and $w$ by (R1).

If $d(u) \geq 6$, then $\frac{\mu(x)}{d(x)} \geq \frac13$ for all $x \in \{u,v,w\}$ and each vertex $u$, $v$, and $w$ sends charge at least $\frac13$, giving $f$ nonnegative final charge.

If $d(u) = 4$, then $d(w) \geq d(v) \geq 7$ since $d(u) + d(v) \geq 11$ by Proposition~\ref{reducible-1}.
If $d(v) \geq 8$, then each of $v$ and $w$ send charge at least $\frac12$, giving $f$ nonnegative final charge. 
Thus, suppose $d(v) = 7$.
Since $d(u) + d(v) = 11$, there is not another 3-face incident to the edge $uv$ by Proposition~\ref{reducible-1}.
Thus, $t_3(v) \leq 6$ and hence $v$ sends charge at least $\frac12$ to $f$.
Similarly, $w$ sends charge at least $\frac12$ so $f$ has nonnegative final charge.

If $d(u) = 5$, then $d(v) \geq 6$ since $d(u) + d(v) \geq 11$ by Proposition~\ref{reducible-1}.
If $d(v) \geq 7$, then vertex $u$ sends charge at least $\frac15$ and each of $v$ and $w$ send charge at least $\frac37$, giving $f$ nonnegative final charge.
If $d(v) = 6$, then there is not another 3-face incident to the edge $uv$ by Proposition~\ref{reducible-1}.
Thus, $t_3(v) \leq 5$ and $v$ sends charge at least $\frac25$. 
Similarly, $w$ sends charge at least $\frac25$ so $f$ has nonnegative final charge.

We conclude that all vertices and faces have nonnegative charge, so $G$ has nonnegative total charge, a contradiction.

\end{proof}

\begin{theorem}\label{thm:47choosable}
If $G$ is a planar graph and does not contain a chorded 4-cycle, then $G$ is $\chunion{4}{7}$-choosable.
\end{theorem}

\begin{proof}
Suppose $G$ is a plane graph minimizing $n(G)+e(G)$ such that $G$ does not contain a chorded 4-cycle and $G$ is not $L$-colorable for some $(4,7)$-list assignment $L$.
By minimality of $G$, we can assume that $d(v) \geq |L(v)| \geq 4$ for all vertices $v$ and $|L(u) \cap L(v)| \geq 1$ for all adjacent pairs $uv$.
In particular, no two adjacent vertices of degree $4$ share a $3$-face by Proposition~\ref{reducible-1}. 
Let the initial charge of a vertex be $d(v)-4$ and that of a face be $\ell(f)-4$. 
Note that the total initial charge of $G$ is $-8$. 
For a vertex $v$, let $t_3(v)$ be the number of 3-faces incident to $v$.
Apply the following discharging rule.
\begin{enumerate}[({R}1)] \itemsep0pc
\item If $v$ is a $5^+$-vertex and $f$ is an incident 3-face, then $v$ sends charge $\frac{\mu(v)}{t_3(v)}$ to $f$.
\end{enumerate}
Since chorded $4$-cycles are forbidden, no two $3$-faces can share an edge. 
Hence for each vertex $v \in V(G)$, there are at most $\lfloor \frac{d(v)}2\rfloor$ $3$-faces incident to $v$.
It follows that vertices of degree at least $5$ send charge at least $\frac12$ to each incident $3$-face. 
Since a $3$-face has at most one incident $4$-vertex by Proposition~\ref{reducible-1}, all $3$-faces have nonnegative final charge after (R1).
Hence all vertices and faces have nonnegative final charge, so $G$ has nonnegative total charge, a contradiction.
\end{proof}

%
%
%
%


\section{$\chunion{3}{11}$-choosability}

\begin{proof}[Proof of Theorem~\ref{thm:311choosable}.]
Suppose $G$ is a plane graph minimizing $n(G)+e(G)$ such that $G$ is not $L$-colorable for some $(3,11)$-list assignment $L$.
By minimality of $G$, we can assume that $d(v) \geq |L(v)| \geq 3$ for all vertices $v$ and $|L(u) \cap L(v)| \geq 1$ for all adjacent pairs $uv$.
By Proposition~\ref{reducible-1},  if $uv$ is an edge in $G$, then $d(u) + d(v) > 11 + \min(|N(u)\cap N(v)|,2)$.
Observe that $\min(|N(u)\cap N(v)|,2)$ is at least the number of 3-faces incident to the edge $uv$.

For each $v \in V(G)$ and $f \in F(G)$ define initial charge functions $\mu(v) = d(v) - 6$ and $\nu(f) = 2\ell(f) - 6$.
By Euler's formula, total charge is $-12$.
Apply the following discharging rules:

\begin{enumerate}[({R}1)] 
\item Let $v$ be a vertex and $u \in N(v)$.
\begin{enumerate}[(a)]
\item If $d(v) = 3$, then $v$ pulls charge $1$ from $u$.
\item If $d(v) = 4$, then $v$ pulls charge $\frac{1}{2}$ from $u$.
\item If $d(v) = 5$, then $v$ pulls charge $\frac{1}{5}$ from $u$.
\end{enumerate}
\item If $f$ is a $4^+$-face and $uv$ is an edge incident to $f$ with $d(u) \leq 5$, then $f$ sends charge $\frac12$ to $v$.
\end{enumerate}

We claim the final charge on all faces and vertices is nonnegative.
Since the total charge sum was preserved during the discharging rules, this contradicts the negative initial charge sum.
Observe that no two $5^-$-vertices are adjacent by Proposition~\ref{reducible-1}, so each face $f$ is incident to at most $\frac{\ell(f)}2$ vertices of degree at most five.
If $f$ is a $3$-face, then $f$ does not lose charge.
If $f$ is a $4^+$-face, then $f$ loses charge at most 1 per incident $5^-$-vertex.
We have $\frac{\ell(f)}2 \leq 2\ell(f)-6$ whenever $\ell(f) \geq 4$, so $f$ has nonnegative final charge.

Each $5^-$-vertex gains exactly enough charge through (R1) so that the final charge is nonnegative.

Suppose $v$ is a $6^+$-vertex.
We introduce some notation to describe the structure near $v$.
For an edge $uv$, let $a(uv)$ be the number of $3$-faces incident to the edge $uv$.
Note that if $d(u) < 6$ and $a(uv) = 0$, then $v$ sends charge at most 1 to $u$ by (R1) and gains charge at least 1 via $uv$ by (R2), giving a nonnegative net difference in charge.
Thus, if $v$ ends with negative charge, it must be due to some number of $5^-$-vertices $u \in N(v)$ with $a(uv) > 0$.

For $k \in \{3,4,5\}$, let $D_k$ be the set of neighbors $u$ of $v$ such that $u$ is a $k$-vertex and $a(uv) = 2$; let $d_k = |D_k|$.
Let $D_3^*$ be the set of neighbors $u$ of $v$ such that $u$ is a $3$-vertex and $a(uv) = 1$; let $d_3^* = |D_3^*|$.
If $u \in D_k$, then $v$ gains no charge via $uv$ in (R2).
If $u \in D_3^*$, then $v$ loses charge 1 to $u$ in (R1) but gains charge $\frac12$ via $uv$ in (R2).
Therefore, the final charge of $v$ at least $\mu(v) - d_3 - \frac12d_3^* - \frac12d_4 - \frac15d_5$.
Recall $\mu(v) = d(v) - 6$, so if $v$ has negative final charge, then
\begin{equation}
d_3 + \frac12d_3^* + \frac12d_4 + \frac15d_5 > d(v) - 6. \label{eq:1}
\end{equation}

Let $D = D_3\cup D_3^* \cup D_4\cup D_5$.
For each $3$-face $uvw$ incident to $v$, at most one of $u,w$ is in $D$.
If $u \in D$, $w \in N(v) \setminus D$, and $uvw$ is a 3-face, then $u$ gives a \emph{strike} to $w$.
Each vertex in $D_3\cup D_4\cup D_5$ contributes two strikes, and each vertex in $D_3^*$ contributes one strike.
The total number of strikes is $2d_3+d_3^*+2d_4+2d_5$ and each vertex $w \in N(v) \setminus D$ receives at most two strikes, so $2d_3 + d_3^* + 2d_4 + 2d_5 \leq 2(d(v) - (d_3+d_3^*+d_4+d_5))$.
Equivalently,
\begin{equation}
2d_3+\frac32 d_3^*+2d_4+2d_5 \leq d(v). \label{eq:2}
\end{equation}
We now have $d(v) \geq 6$ and the two inequalities (\ref{eq:1}) and (\ref{eq:2}).
Also recall that since $d(u) + d(v) > 11 + \min(|N(u) \cap N(v)|,2)$, we have the following implications: if $d(v) \leq 10$ then $d_3 = 0$; if $d(v) \leq 9$, then $d_3^* + d_4 = 0$; if $d(v) \leq 8$, then $d_5 = 0$.

If we subtract (\ref{eq:1}) from (\ref{eq:2}), then we find the following inequality.
\begin{equation}
d_3 + d_3^* + \frac{3}{2}d_4 + \frac{9}{5}d_5 < 6. \label{eq:3}
\end{equation}
There are 77 tuples $(d_3,d_3^*, d_4, d_5)$ of nonnegative integers that satisfy (\ref{eq:3}); see Appendix~\ref{sec:tuples} for the full list.
None of these tuples admit a value $d(v)$ that satisfies (\ref{eq:1}) and the implications.
Therefore, there is no $6^+$-vertex $v$ with negative final charge.
We conclude that all vertices and faces have nonnegative final charge.
But total charge is $-12$, a contradiction.
Thus a minimum counterexample does not exist and all planar graphs are $\chunion{3}{11}$-choosable.
\end{proof}


%


{\small
\bibliographystyle{abbrv}
\bibliography{choosability}

\begin{thebibliography}{10}

\bibitem{DMWS14}
Z.~Berikkyzy, C.~Cox, M.~Dairyko, K.~Hogenson, M.~Kumbhat, B.~Lidick\'y,
  K.~Messerschmidt, K.~Moss, K.~Nowak, K.~F. Palmowski, and D.~Stolee.
\newblock $(4,2)$-choosability of planar graphs with forbidden structures,
  2015.
\newblock Available as arXiv:1512.03787 [math.CO].

\bibitem{BorodinSurvey}
O.~V. Borodin.
\newblock Colorings of plane graphs: a survey.
\newblock {\em Discrete Mathematics}, 313(4):517--539, 2013.

\bibitem{borodin}
O.~V. Borodin and A.~O. Ivanova.
\newblock Planar graphs without triangular 4-cycles are 4-choosable.
\newblock {\em Sib. \'Elektron. Mat. Izv.}, 5:75--79, 2008.

\bibitem{choi}
I.~Choi, B.~Lidick\'y, and D.~Stolee.
\newblock On choosability with separation of planar graphs with forbidden
  cycles.
\newblock {\em Journal of Graph Theory}.
\newblock to appear.

\bibitem{CW}
D.~Cranston and D.~B. West.
\newblock A guide to the discharging method.
\newblock Available as arXiv:1306.4434 [math.CO].

\bibitem{erdos1979choosability}
P.~Erd{\"o}s, A.~L. Rubin, and H.~Taylor.
\newblock Choosability in graphs.
\newblock {\em Congr. Numer}, 26:125--157, 1979.

\bibitem{furedi}
Z.~F{\"u}redi, A.~Kostochka, and M.~Kumbhat.
\newblock Choosability with separation of complete multipartite graphs and
  hypergraphs.
\newblock {\em Journal of Graph Theory}, 76(2):129--137, 2014.

\bibitem{kl}
H.~Kierstead and B.~Lidick\'y.
\newblock On choosability with separation of planar graphs with lists of
  different sizes.
\newblock {\em Discrete Mathematics}, 339(10):1779--1783, 2015.

\bibitem{krat2}
J.~Kratochv\'il and Z.~Tuza.
\newblock Algorithmic complexity of list colorings.
\newblock {\em Discrete Applied Mathematics}, 50(3):297--302, 1994.

\bibitem{krat}
J.~Kratochv\'il, Z.~Tuza, and M.~Voigt.
\newblock Brooks-type theorems for choosability with separation.
\newblock {\em Journal of Graph Theory}, 27(1):43--49, 1998.

\bibitem{thomassen}
C.~Thomassen.
\newblock Every planar graph is 5-choosable.
\newblock {\em Journal of Combinatorial Theory, Series B}, 62(1):180--181,
  1994.

\bibitem{Vizing}
V.~G. Vizing.
\newblock Coloring the vertices of a graph in prescribed colors.
\newblock {\em Diskret. Analiz}, 29:3--10, 1976.
\newblock (In Russian).

\bibitem{voigt1993list}
M.~Voigt.
\newblock List colourings of planar graphs.
\newblock {\em Discrete Mathematics}, 120(1):215--219, 1993.

\bibitem{voigt}
M.~Voigt.
\newblock A not 3-choosable planar graph without 3-cycles.
\newblock {\em Discrete Mathematics}, 146(1-3):325--328, 1995.

\bibitem{skrekovski}
R.~\v{S}krekovski.
\newblock A note on choosability with separation for planar graphs.
\newblock {\em Ars Combinatoria}, 58:169--174, 2001.

\end{thebibliography}
}

\clearpage
\appendix
\section{List of Tuples}\label{sec:tuples}

The following list of tuples $(d_3,d_3^*,d_4,d_5)$ satisfy inequality (\ref{eq:3}).
Recall that $d(v) \geq 6$, $d_5 > 0$ implies $d(v) \geq 9$, $d_3^*+d_4 > 0$ implies $d(v) \geq 10$, and $d_3 > 0$ implies $d(v) \geq 11$.
After these implications are applied, we find that the tuple $(d_3,d_3^*, d_4,d_5,d(v))$ violates inequality (\ref{eq:1}).

\begin{multicols}{3}
\footnotesize\noindent
$(0,0,0,0)$ fails (\ref{eq:1}) for $d(v) \geq 6$.\\
$(0,0,0,1)$ fails (\ref{eq:1}) for $d(v) \geq 9$.\\
$(0,0,0,2)$ fails (\ref{eq:1}) for $d(v) \geq 9$.\\
$(0,0,0,3)$ fails (\ref{eq:1}) for $d(v) \geq 9$.\\
$(0,0,1,0)$ fails (\ref{eq:1}) for $d(v) \geq 10$.\\
$(0,0,1,1)$ fails (\ref{eq:1}) for $d(v) \geq 10$.\\
$(0,0,1,2)$ fails (\ref{eq:1}) for $d(v) \geq 10$.\\
$(0,0,2,0)$ fails (\ref{eq:1}) for $d(v) \geq 10$.\\
$(0,0,2,1)$ fails (\ref{eq:1}) for $d(v) \geq 10$.\\
$(0,0,3,0)$ fails (\ref{eq:1}) for $d(v) \geq 10$.\\
$(0,1,0,0)$ fails (\ref{eq:1}) for $d(v) \geq 10$.\\
$(0,1,0,1)$ fails (\ref{eq:1}) for $d(v) \geq 10$.\\
$(0,1,0,2)$ fails (\ref{eq:1}) for $d(v) \geq 10$.\\
$(0,1,1,0)$ fails (\ref{eq:1}) for $d(v) \geq 10$.\\
$(0,1,1,1)$ fails (\ref{eq:1}) for $d(v) \geq 10$.\\
$(0,1,2,0)$ fails (\ref{eq:1}) for $d(v) \geq 10$.\\
$(0,1,2,1)$ fails (\ref{eq:1}) for $d(v) \geq 10$.\\
$(0,1,3,0)$ fails (\ref{eq:1}) for $d(v) \geq 10$.\\
$(0,2,0,0)$ fails (\ref{eq:1}) for $d(v) \geq 10$.\\
$(0,2,0,1)$ fails (\ref{eq:1}) for $d(v) \geq 10$.\\
$(0,2,0,2)$ fails (\ref{eq:1}) for $d(v) \geq 10$.\\
$(0,2,1,0)$ fails (\ref{eq:1}) for $d(v) \geq 10$.\\
$(0,2,1,1)$ fails (\ref{eq:1}) for $d(v) \geq 10$.\\
$(0,2,2,0)$ fails (\ref{eq:1}) for $d(v) \geq 10$.\\
$(0,3,0,0)$ fails (\ref{eq:1}) for $d(v) \geq 10$.\\
$(0,3,0,1)$ fails (\ref{eq:1}) for $d(v) \geq 10$.\\
$(0,3,1,0)$ fails (\ref{eq:1}) for $d(v) \geq 10$.\\
$(0,4,0,0)$ fails (\ref{eq:1}) for $d(v) \geq 10$.\\
$(0,4,0,1)$ fails (\ref{eq:1}) for $d(v) \geq 10$.\\
$(0,4,1,0)$ fails (\ref{eq:1}) for $d(v) \geq 10$.\\
$(0,5,0,0)$ fails (\ref{eq:1}) for $d(v) \geq 10$.\\
$(1,0,0,0)$ fails (\ref{eq:1}) for $d(v) \geq 11$.\\
$(1,0,0,1)$ fails (\ref{eq:1}) for $d(v) \geq 11$.\\
$(1,0,0,2)$ fails (\ref{eq:1}) for $d(v) \geq 11$.\\
$(1,0,1,0)$ fails (\ref{eq:1}) for $d(v) \geq 11$.\\
$(1,0,1,1)$ fails (\ref{eq:1}) for $d(v) \geq 11$.\\
$(1,0,2,0)$ fails (\ref{eq:1}) for $d(v) \geq 11$.\\
$(1,0,2,1)$ fails (\ref{eq:1}) for $d(v) \geq 11$.\\
$(1,0,3,0)$ fails (\ref{eq:1}) for $d(v) \geq 11$.\\
$(1,1,0,0)$ fails (\ref{eq:1}) for $d(v) \geq 11$.\\
$(1,1,0,1)$ fails (\ref{eq:1}) for $d(v) \geq 11$.\\
$(1,1,0,2)$ fails (\ref{eq:1}) for $d(v) \geq 11$.\\
$(1,1,1,0)$ fails (\ref{eq:1}) for $d(v) \geq 11$.\\
$(1,1,1,1)$ fails (\ref{eq:1}) for $d(v) \geq 11$.\\
$(1,1,2,0)$ fails (\ref{eq:1}) for $d(v) \geq 11$.\\
$(1,2,0,0)$ fails (\ref{eq:1}) for $d(v) \geq 11$.\\
$(1,2,0,1)$ fails (\ref{eq:1}) for $d(v) \geq 11$.\\
$(1,2,1,0)$ fails (\ref{eq:1}) for $d(v) \geq 11$.\\
$(1,3,0,0)$ fails (\ref{eq:1}) for $d(v) \geq 11$.\\
$(1,3,0,1)$ fails (\ref{eq:1}) for $d(v) \geq 11$.\\
$(1,3,1,0)$ fails (\ref{eq:1}) for $d(v) \geq 11$.\\
$(1,4,0,0)$ fails (\ref{eq:1}) for $d(v) \geq 11$.\\
$(2,0,0,0)$ fails (\ref{eq:1}) for $d(v) \geq 11$.\\
$(2,0,0,1)$ fails (\ref{eq:1}) for $d(v) \geq 11$.\\
$(2,0,0,2)$ fails (\ref{eq:1}) for $d(v) \geq 11$.\\
$(2,0,1,0)$ fails (\ref{eq:1}) for $d(v) \geq 11$.\\
$(2,0,1,1)$ fails (\ref{eq:1}) for $d(v) \geq 11$.\\
$(2,0,2,0)$ fails (\ref{eq:1}) for $d(v) \geq 11$.\\
$(2,1,0,0)$ fails (\ref{eq:1}) for $d(v) \geq 11$.\\
$(2,1,0,1)$ fails (\ref{eq:1}) for $d(v) \geq 11$.\\
$(2,1,1,0)$ fails (\ref{eq:1}) for $d(v) \geq 11$.\\
$(2,2,0,0)$ fails (\ref{eq:1}) for $d(v) \geq 11$.\\
$(2,2,0,1)$ fails (\ref{eq:1}) for $d(v) \geq 11$.\\
$(2,2,1,0)$ fails (\ref{eq:1}) for $d(v) \geq 11$.\\
$(2,3,0,0)$ fails (\ref{eq:1}) for $d(v) \geq 11$.\\
$(3,0,0,0)$ fails (\ref{eq:1}) for $d(v) \geq 11$.\\
$(3,0,0,1)$ fails (\ref{eq:1}) for $d(v) \geq 11$.\\
$(3,0,1,0)$ fails (\ref{eq:1}) for $d(v) \geq 11$.\\
$(3,1,0,0)$ fails (\ref{eq:1}) for $d(v) \geq 11$.\\
$(3,1,0,1)$ fails (\ref{eq:1}) for $d(v) \geq 11$.\\
$(3,1,1,0)$ fails (\ref{eq:1}) for $d(v) \geq 11$.\\
$(3,2,0,0)$ fails (\ref{eq:1}) for $d(v) \geq 11$.\\
$(4,0,0,0)$ fails (\ref{eq:1}) for $d(v) \geq 11$.\\
$(4,0,0,1)$ fails (\ref{eq:1}) for $d(v) \geq 11$.\\
$(4,0,1,0)$ fails (\ref{eq:1}) for $d(v) \geq 11$.\\
$(4,1,0,0)$ fails (\ref{eq:1}) for $d(v) \geq 11$.\\
$(5,0,0,0)$ fails (\ref{eq:1}) for $d(v) \geq 11$.\\
\end{multicols}

\end{document}